\documentclass[11pt,reqno]{amsart}

\usepackage{setspace}
\usepackage{fullpage}
\usepackage{datetime}
\usepackage{xcolor}
\usepackage{caption}

\usepackage{enumitem}
\newcommand\rmlabel{\upshape({\itshape\roman*\,\/})}

\usepackage{amsmath}%
\usepackage{amsfonts}%
\usepackage{amssymb}%
\usepackage{graphicx}
\usepackage{mathrsfs}
\usepackage{hyperref}
\usepackage[margin=1in]{geometry}
\usepackage[normalem]{ulem}
\usepackage{mathtools}

\usepackage{tikz}
\usetikzlibrary{fit,positioning,backgrounds}

\allowdisplaybreaks[4] 
\newtheorem{theorem}{Theorem}
\theoremstyle{plain}

\newtheorem{claim}[theorem]{Claim}

\newtheorem{conjecture}[theorem]{Conjecture}

\newtheorem{definition}[theorem]{Definition}

\newtheorem{lemma}[theorem]{Lemma}

\newtheorem{proposition}[theorem]{Proposition}

\newtheorem{setup}{Setup}

\numberwithin{equation}{section}
\numberwithin{theorem}{section}
\numberwithin{case}{section}

\newtheorem*{lemma*}{Lemma~\ref{lem:abs}}
\newtheorem*{Lemma*}{Lemma~\ref{lem:con}}
\newtheorem*{LemmA*}{Lemma~\ref{lem:cov}}
\newtheorem*{lemmA*}{Lemma~\ref{lem:res}}

\usepackage{euflag}

\usepackage{pgfplots} 
\usepackage{tikz}
\usetikzlibrary{calc}
\usetikzlibrary{arrows,shapes,chains}
\pgfdeclarelayer{background}
\pgfdeclarelayer{foreground}
\pgfsetlayers{background,main,foreground}
\usetikzlibrary{decorations.pathreplacing}

\def\COMMENT#1{}
\let\COMMENT=\footnote
\usepackage{lineno} 

\begin{document}

\title{Clique-factors in randomly perturbed hypergraphs}

\author[Y.~Chang]{Yulin Chang}
\address{(Y. Chang) School of Mathematical Sciences, Ocean University of China, Qingdao, 266100, China}
\email{ylchang@ouc.edu.cn}

\author[J.~Han]{Jie Han}
\address{(J. Han) School of Mathematics and Statistics, Beijing Institute of Technology, Beijing, 100000, China}
\email{han.jie@bit.edu.cn}

\author[P.~Morris]{Patrick Morris}
\address{(P. Morris) Departament de Matemàtiques and IMTECH, Universitat Politècnica de Catalunya (UPC), Barcelona, 08034, Spain}
\email{pmorrismaths@gmail.com}

\author[L.~Sun]{Lin Sun}
\address{(L. Sun) School of Mathematics, Qingdao University, Qingdao, 266071, China}
\email{sun.lin@qdu.edu.cn}

\thanks{Y.~Chang is partially supported by Natural Science Foundation of Shandong Province (ZR2026MS0029) and Natural Science Foundation of China (12201352). J.~Han is partially supported by Natural Science Foundation of China (12371341). P.~Morris was supported by PCI2024-155080-2 (SRC-ExCo), by the
 European Union's Horizon Europe Marie Sk{\l}odowska-Curie grant
 RAND-COMB-DESIGN - project number 101106032 {\euflag} and by a Ram\'on y
 Cajal fellowship RYC2024-049272-I.  L.~Sun is partially supported by  Natural Science Foundation of China (12601674), Natural Science Foundation of Shandong Province (ZR2026QC0057) and  Qingdao
 Natural Science Foundation (26-1-1-68-zyyd-jch).}

\begin{abstract}
In this paper we establish the following result for clique-factors in randomly perturbed hypergraphs when the cliques are large enough.
For integers $k \geq 3$ and $t\geq t_0(k)$, we show that for every $\alpha> 0$, there exists $C>0$ such that for sufficiently large $n\in(t+1)\mathbb{N}$, the union of a $k$-uniform hypergraph on $n$ vertices with minimum codegree at least $\alpha n$ and a binomial random $k$-uniform hypergraph $H^{(k)}(n,p)$ with $p\geq Cn^{-{t/\binom{t+1}{k}}}$ on the same vertex set contains a $K_{t+1}^{(k)}$-factor asymptotically almost surely.
The bound on $p$ is asymptotically best possible and extends the graph case, resolved by Balogh, Treglown and Wagner.
In particular, we have $t_0(3)=3$, which implies that our result applies to every $t\ge 3$ for $3$-uniform hypergraphs. Together with the previously established result for perfect matchings, this covers all possible clique sizes for $3$-uniform hypergraphs.

\end{abstract}


\maketitle

\section{Introduction}
\label{s1}

Let $k \ge 2$ be an integer.
A $k$-uniform hypergraph (or $k$-graph) $H$ consists of a vertex set $V$ and an edge set $E \subseteq \binom{V}{k}$, where $\binom{V}{k}$ denotes the family of all $k$-subsets of $V$. If $E = \binom{V}{k}$, $H$ is called the \emph{complete $k$-graph} and is denoted by $K_n^{(k)}$, where $n = |V|$. When $k=2$, this simply corresponds to a standard \emph{graph}.
Given two $k$-graphs $F$ and $H$, an \emph{$F$-tiling} in $H$ is defined as a collection of vertex-disjoint copies of $F$. 
An $F$-tiling that covers all vertices of $H$ is referred to as a \emph{perfect $F$-tiling}, which is also known as an \emph{$F$-factor} or a \emph{perfect $F$-packing}.
A necessary condition for its existence is that $|V(F)|$ divides $|V(H)|$.
A fundamental objective in extremal graph theory is to establish various conditions that guarantee the existence of an  $F$-factor in a given $k$-graph $H$. 
When $F$ has the same  number of vertices  as $H$, an $F$-factor is simply a spanning copy of $F$ in $H$, with prominent examples including Hamilton cycles, powers of Hamilton cycles, and  bounded-degree spanning trees. 
Alternatively, $F$ may be a fixed hypergraph, such as a clique of a specified size.
In this paper, we focus on the latter  setting, where $F = K_t^{(k)}$ for a fixed integer $t\ge k$; that is, $F$ is the complete $k$-graph on $t$ vertices.
For any positive integer 
$n$, let $[n]\coloneqq\{1,2,\dots,n\}$.

\subsection{Graphs and hypergraphs}
In the graph case ($k=2$), the celebrated Hajnal--Szemer\'{e}di theorem~\cite{HS} asserts that any graph on $n$ vertices with a minimum degree at least $(1 -1/t)n$ contains a $K_t$-factor.  For arbitrary graphs $F$, K\"{u}hn and Osthus~\cite{KO} determined, up to an additive constant, the minimum degree required for a graph to guarantee an $F$-factor, improving upon earlier results in~\cite{AY,KSS}.
Extending such degree conditions to $k$-graphs is a natural and important research direction.
For a $k$-graph $H$ and a $d$-subset $S\subseteq V(H)$ with $1\leq d\leq k-1$, the \emph{degree} of $S$ is defined as $\deg_H(S)\coloneqq\left|\left\{e\in E(H)\colon S\subseteq e\right\}\right|$.
The \emph{minimum $d$-degree} $\delta_{d}(H)$ of $H$ is the minimum of $\deg_H(S)$ over all $d$-subsets $S$ of $V(H)$.
In particular, $\delta_{k-1}(H)$ is called the \emph{minimum codegree}, and $\delta_1(H)$ is the \emph{minimum vertex degree}. 
However, determining these exact minimum degree requirements becomes significantly more challenging for hypergraphs, even for small $k$-graphs $F$.
For instance, the exact minimum vertex degree required to ensure the existence of a perfect matching in $k$-graphs remains unknown for $k\ge 6$.
For clique factors, the current best general minimum codegree conditions are due to Lo and Markstr\"om~\cite{lo2015f}. For arbitrary $t$, they showed that 
any $n$-vertex $3$-graph $H$ with $\delta_2(H) \ge \left(1-\frac{2}{t^2-3t+4}+o(1)\right)n$ contains a $K_t^{(3)}$-factor. For the specific case of $K_4^{(3)}$-factors, Lo and Markstr\"om~\cite{lo2015f} showed that a minimum codegree of $3n/4+o(n)$ is sufficient by using the absorption method. Independently and simultaneously, Keevash and Mycroft~\cite{KM2015} determined the exact minimum codegree requirement. They proved that any $n$-vertex $3$-graph $H$ contains a $K_4^{(3)}$-factor if $\delta_2(H) \ge 3n/4-2$ when $n \in 8\mathbb{N}$, and $\delta_2(H) \ge 3n/4-1$ otherwise.
For further results on factors in graphs and hypergraphs, we refer the reader to the surveys~\cite{KO1,RR,Zhao}.

\subsection{Random (hyper)graphs}
Similar questions have been studied in the setting of the binomial random $k$-graph $H^{(k)}(n, p)$, which contains $n$ vertices
and each $k$-subset forms an edge independently with probability $p=p(n)$.
A \emph{property} $\mathcal{P}$ is a subset of the set of all labeled $k$-graphs on the vertex set $[n]$. 
We are particularly interested in \emph{monotone increasing} properties, such as containing an 
$F$-factor.
We say that $H^{(k)}(n,p)$ has the property $\mathcal{P}$ \emph{asymptotically almost surely} (a.a.s.) if $ \lim_{n\rightarrow \infty} \mathbb{P}\left[ H^{(k)}(n,p)\in\mathcal{P}\right] =1$. 
A function $p^* = p^*(n)$ is called a \emph{threshold} for a monotone increasing property $\mathcal{P}$ if,
a.a.s., $H^{(k)}(n,p)\in\mathcal{P}$ when $p = \omega(p^*)$, and $H^{(k)}(n,p)\notin\mathcal{P}$
when $p = o(p^*)$.
The threshold for the appearance of an 
$F$-factor in $H^{(k)}(n, p)$ depends on specific density parameters of $F$.
For a $k$-graph $F$ with at least two vertices, let $v_F$ and $e_F$ denote the number of vertices and edges in~$F$, respectively, and define 
\begin{equation*}
    d(F)\coloneqq \frac{e_{F}}{v_{F}-1} \mbox{ and } d^*(F)\coloneqq \max\left\{d(F')\colon F'\subseteq F, v_{F'}\ge 2\right\}.
\end{equation*}

Johansson, Kahn, and Vu~\cite{JKV} made significant progress on this problem for both graphs and hypergraphs.
They conjectured that the threshold for the appearance of an $F$-factor in a binomial random $k$-graph is $\ell(n;F)n^{-1/d^*(F)}$, where $\ell(n;F)$ is an explicit polylogarithmic factor that depends on the structure of~$F$ (see~\cite{JKV} for details).
Furthermore, they proved that the conjecture holds when the~$\ell(n;F)$ term is replaced by $n^{o(1)}$, and determined the threshold (where $\ell(n;F)=(\log n)^{1/e_F}$) for all strictly balanced $k$-graphs~$F$.
In particular, for $F=K_t^{(k)}$ with $t\ge k$, this implies that the threshold is $n^{-(t-1)/\binom{t}{k}}(\log n)^{1/\binom{t}{k}}$. 
Subsequently,  for ordinary graphs ($k=2$), Gerke and McDowell~\cite{GM} determined the threshold (where $\ell(n;F)=1$) for non-vertex-balanced graphs $F$.

\subsection{Randomly perturbed (hyper)graphs}
For two $k$-graphs $H$ and $H'$, let $H\cup H'$ denote the $k$-graph with vertex set $V(H)\cup V(H')$ and edge set $E(H)\cup E(H')$. In this paper, we study the model of a randomly perturbed hypergraph $H\cup H^{(k)}(n,p)$, introduced by Bohman, Frieze, and Martin~\cite{BFM}.
This model starts with a deterministic $n$-vertex $k$-graph $H$ with certain density condition
and adds to it the edges of a random $k$-graph $H^{(k)}(n,p)$ on the same vertex set.

In the graph case ($k=2$), Balogh, Treglown, and Wagner~\cite{BTW} determined, for any fixed graph $F$, the number of random edges one needs to add to a graph $H$ with linear minimum degree to ensure that the resulting graph contains an $F$-factor a.a.s. We state their result formally as follows.
\begin{theorem}[Balogh, Treglown, and Wagner~\cite{BTW}]
Let $F$ be a fixed graph and let $\alpha > 0$. There exists $C > 0$ such that the following holds for sufficiently large $n \in |V(F)|\mathbb{N}$. Let $G$ be a graph on $n$ vertices with $\delta(G) \ge \alpha n$, and let $p=p(n) \ge C n^{-1/d^*(F)}$. Then a.a.s.~$G \cup G(n,p)$ contains an $F$-factor.
\end{theorem}

A strengthening of this result was recently obtained in \cite{han2025perturbation}, in which the minimum degree condition was relaxed to a density condition for $F$ with $d^*(F) > 1$.
For hypergraphs ($k\ge 3$),
Krivelevich, Kwan, and Sudakov~\cite{KKS2} considered perfect matchings under a minimum codegree condition.
These results were later extended by Chang, Han, Kohayakawa, Morris, and Mota~\cite{chang2022factors}, who determined the required amount of random edges for a
certain class of $k$-graphs $F$ (including all $k$-partite $k$-graphs, $K_4^{(3)-}$, and the Fano plane) under a much weaker vertex degree condition. 
In this paper, we focus on the case of clique factors.


To formulate our main result, we first introduce two technical setups that capture the inequalities required in our arguments.
These conditions determine the range of parameters $(k,t)$ for which our method applies. Let $k\geq 3$ and $t\geq k$ be integers.

\begin{setup}\label{setup1} 
We say that $(k,t)$ satisfy Setup~\ref{setup1} if there exists some $0<\varepsilon<\min\left\{1,t/\binom{t+1}{k}\right\}$ such that \begin{align*}\binom{t}{k-1} \le 2 \left(\left\lceil\frac{t}{k-1}\right\rceil\right)^{\ell}\left(\left\lfloor\frac{t}{k-1}\right\rfloor\right)^{k-1-\ell}+\binom{t+1}{k}\frac{1-\varepsilon}{t},\end{align*} where $\ell\equiv t \pmod{k-1}$ is the unique integer with $0\le \ell\le k-2$.
\end{setup}

\begin{setup}\label{setup2} We say that $(k,t)$ satisfy Setup~\ref{setup2} if there exists some  
$0<\varepsilon<\min\left\{1,t/\binom{t+1}{k}\right\}$ such that 
\begin{align*}
 \binom{t}{k-1} \le  \left(\left\lceil\frac{t+1}{k}\right\rceil\right)^{m}\left(\left\lfloor\frac{t+1}{k}\right\rfloor\right)^{k-m}+\left(\left\lceil\frac{t+1}{k}\right\rceil\right)^{m-1}\left(\left\lfloor\frac{t+1}{k}\right\rfloor\right)^{k-m}+\binom{t+1}{k}\frac{1-\varepsilon}{t},\label{equ:condition2}
\end{align*} where $m\equiv (t+1) \pmod{k}$ is the unique integer with $0\le m\le k-1$.
\end{setup}

\textbf{Remark:}  Both  Setup~\ref{setup1} and Setup~\ref{setup2}  hold for every $t\ge3$ in the case $k=3$. For larger values of $k$, the validity of these conditions depends on the pair $(k,t)$. The admissible parameter ranges are summarized in Table~\ref{tab:summary}. 

\begin{table}[htbp]
\centering
\renewcommand{\arraystretch}{1.6}  
\begin{tabular}{|c|c|c|}
\hline
\multicolumn{1}{|c|}{} & \multicolumn{1}{c|}{\textbf{Setup~\ref{setup1}}} & \multicolumn{1}{c|}{\textbf{Setup~\ref{setup2}}} \\
\hline
$k = 3$ & $t \geq 3$ & $t \geq 3$ \\
\hline
$k = 4$ & $t \in \{  4, 5, 6\}$ & $t = 4,\ t \geq 19$ \\
\hline
$k = 5$ & $t =5$ & $ t \geq 87$ \\
\hline
$k \ge 6$ & $-$ & $ t \geq t_0(k)$ \\
\hline
\end{tabular}
\caption{Valid Cases for Setup~\ref{setup1} and Setup~\ref{setup2}.}
\label{tab:summary}
\end{table}

For any $(k,t)$ satisfying Setup~\ref{setup1} or~\ref{setup2}, we fix some choice of $\varepsilon=\varepsilon(k,t)$ which suffices for the defining inequality. We are now in a position to state our main theorem, which establishes the random perturbation threshold for the existence of a $K_{t+1}^{(k)}$-factor in dense $k$-graphs.

\begin{theorem}[Main result]\label{main}
Let  $t\geq k\geq 3$ be  integers such that one of the following holds:
\begin{itemize}
\item[$(i)$] $(k,t) \in \{((4,4),(4,5),(4,6),(5,5)\}$;
\item[$(ii)$] $t \ge t_0(k)$ for some $t_0(k)$, where $t_0(3)=3$.
\end{itemize} 
For every $\alpha>0$, there exists $C>0$ such that the following holds for sufficiently large $n\in(t+1)\mathbb{N}$.
Let $H$ be a $k$-graph on $n$ vertices with $\delta_{k-1}(H)\geq \alpha n$, and let $p= p(n)\geq Cn^{-{t/\binom{t+1}{k}}}$.
Then a.a.s.~$H\cup H^{(k)}(n,p)$ contains a $K_{t+1}^{(k)}$-factor.
\end{theorem}

We remark that Theorem~\ref{main} could be extended to cover all cases where $t=k-1$. In this case,  a $K_{t+1}^{(k)}$-factor is a perfect matching and the result was proved already by Chang, Han, Kohayakawa, Morris, and Mota~\cite{chang2022factors}. We also remark that Theorem~\ref{main} is tight in the sense that for all $0<\alpha<1/(t+1)$, there is some constant $c=c(\alpha)$ and an $n$-vertex $k$-graph $H$ with $\delta_{k-1}(H)\geq \alpha n$   such that if $p\leq cn^{-{t/\binom{t+1}{k}}}$ then a.a.s.\ $H\cup H^{(k)}(n,p)$ does \textit{not} contain a $K_{t+1}^{(k)}$-factor.  Indeed, one can take $H$ to be the hypergraph consisting of all edges touching some vertex set $U\subseteq V(H)$ of size $\alpha n$. In a $K_{t+1}^{(k)}$-factor, at most $|U|$ cliques could contain an edge of $H$ and so there would have to be $(1/(t+1)-\alpha)n$ copies of $K_{t+1}^{(k)}$ entirely contained in $H^{(k)}(n,p)$. A simple application of Chebyshev's inequality shows that a.a.s.\ this does not occur for an appropriately chosen constant $c>0$, see for example \cite[Theorem 4.9]{Janson}. 

Let us briefly outline the (very) high-level idea that underpins our proof and how we use both the deterministic and the random graph.  Write $K:=K_{t+1}^{(k)}$. 
A typical application of the absorption method needs the following: for a given vertex $v$, we want to control the number of sets $S\dot\cup \{w\}$ such that both $S\cup \{v\}$ and $S\cup \{w\}$ induce a copy of $K$.
To construct such structures, for the graph case, one can take $w$ that shares many common neighbors with $v$ in $H$, and choose a $t$-clique using random edges in their common neighborhoods.
This is no longer possible for $k\ge 3$.
Our strategy is to choose $w$ such that $v$ and $w$ are in many copies of (almost) balanced complete $k$-partite $k$-graphs, and then fill in the missing edges by random edges.
For large $t$, this saves roughly a $k!/k^{k}$ portion of edges which compensates the loss due to the fact that $v$ is fixed.
Nevertheless, this indeed covers all values of $t$ for $k=3$ and all but a constant number of $t$ for $k\ge 4$.

\medskip
\textbf{Organization:} 
In Section~\ref{s2}, we introduce some notation and the proof strategy for our main theorem.
In Section~\ref{s3}, we establish several preliminary results for random hypergraphs, which will be used in the proof of the absorbing lemma.
In Section~\ref{s4}, we prove our absorbing lemma.
Finally, we present a proof of the main theorem in Section~\ref{s5}.

\section{Overview and Notation}\label{s2}
\subsection{Proof ideas}
To find a  $K_{t+1}^{(k)}$-factor in~$H\cup H^{(k)}(n,p)$,  we use the ``absorption method'' pioneered by R\"{o}dl, Ruci\'nski, and Szemer\'{e}di~\cite{Rodl1}.  
 The proof consists of three main steps. First, we construct a small absorbing set. Next, after removing the absorbing set, we use the random hypergraph to obtain a collection of vertex-disjoint copies of $K_{t+1}^{(k)}$ covering the majority of vertices. Finally, the absorbing set is used to absorb the remaining vertices, completing a $K_{t+1}^{(k)}$-factor.

The main difficulty lies in constructing the absorbing set.
To apply the absorbing lemma from \cite{han2025perturbation}, which is based on Montgomery's bipartite-template method  \cite{M}, it suffices to construct many absorbers for every $(t+1)$-set $W$. For this purpose, we introduce a ``comb $\mathcal{C}_W$ chain'' rooted at $W$ as a $W$-absorber. 
Rather than constructing copies of $K_{t+1}^{(k)}$ directly, we first build suitable deterministic structures using only the dense hypergraph $H$. 
The minimum degree condition on $H$ implies that most vertices are good, meaning that for each such vertex, there are many other vertices sharing many common neighbors with it.
For every good vertex $v$, there are many partite graphs containing $v$ together with another  vertex.
These partite graphs serve as deterministic skeletons of the comb.

The random perturbation is then used to complete the missing edges of these deterministic skeletons, thereby turning them into copies of $K_{t+1}^{(k)}$. Depending on the admissible parameter pair $(k,t)$, we employ one of two deterministic constructions, corresponding to Setup~\ref{setup1} or Setup~\ref{setup2}. 
After verifying the required density conditions for the associated random parts, Proposition~\ref{p1} yields many vertex-disjoint embeddings of the random structures, which in turn produce the desired comb absorbers. 
This establishes the absorbing lemma and, together with the covering lemma, completes the proof of Theorem~\ref{main}.

\subsection{Notation}
In this subsection we introduce some notation used throughout the paper.
Given a $k$-graph $H=(V,E)$, we use $v_H$ and $e_H$ to denote the number of vertices and edges of $H$, respectively.
For an arbitrary vertex subset $U\subseteq V$, we denote by $H-U$ the induced subgraph $H[V \setminus U]$.
When $U = \{v\}$ is a singleton, we simply write $H-v$ instead of $H-\{v\}$.
The \emph{complement} $\overline{H}$ of $H$ is the $k$-graph on $V$ with edge set $\binom{V}{k}\setminus E$. 
For a vertex $v\in V$, its   \emph{neighborhood} is defined as $N_H(v): = \{ S \subseteq V : S \cup \{v\} \in E \}$.
The \emph{link} $L_v$ of $v$ is the $(k-1)$-graph with vertex set $V(L_v)=V\setminus\{v\}$ and edge set $E(L_v)=\left\{X\colon\,X\cup \{v\}\in E\right\}$.
A \emph{path} in $H$ from a vertex $u$ to a vertex $v$ is an alternating sequence of distinct vertices and distinct edges $u=v_0,e_1,v_1,e_2,v_2,\dots,e_m,v_m=v$ such that $v_{i-1},v_i\in e_i$ for each $i\in[m]$.
The $k$-graph $H=(V,E)$ is said to be \emph{connected} if for every pair of distinct vertices $u,v\in V$, there exists a path connecting them.
For any positive integers 
$m$ and $n$ with $m<n$, let $[m,n]$ denote the set of integers $\{m,m+1,\dots,n\}$.

Throughout the paper, we assume that $n$ is sufficiently large unless
stated otherwise, and we omit floor and ceiling functions unless necessary.
We write $\alpha\ll \beta\ll \gamma$ to mean that we can choose positive constants $\alpha,\beta,\gamma$ from right to left.
More precisely, there are two monotone functions $f$ and $g$ such that, given $\gamma$, whenever $\beta\le f(\gamma)$ and $\alpha\le g(\beta)$, the subsequent statements hold.
Hierarchies with more constants are defined analogously.
Implicitly, we assume that all constants appearing in a hierarchy are positive.

\section{Preliminaries}
\label{s3}
In this section, we collect some useful results on binomial random $k$-graphs $H^{(k)}(n, p)$.
\subsection{Rooted Subgraph counts in random hypergraphs}
Following~\cite{HMT2021,Janson}, given a $k$-graph $F$, we define $\Phi_F=\Phi_F(n,p):=\min\left\{n^{v_H}p^{e_H}\colon\,H\subseteq F,\, e_H>0\right\}$.
Here, we focus on the appearance of $k$-graphs in $H^{(k)}(n, p)$ where a subset of vertices may be fixed in place.
Therefore, for a $k$-graph $F$ and a subset of \emph{independent} vertices $W \subseteq V (F)$ (i.e.~$E(F[W])=\emptyset$), we define
\[
\Phi_{F,W}=\Phi_{F,W}(n,p):=\min\left\{n^{v_H-v_{H[W]}}p^{e_H}\colon\,H\subseteq F,\, e_H>0\right\}.
\]
It is easy to see that $\Phi_F=\Phi_{F,\emptyset}$ and $\Phi_{F\setminus W}\geq \Phi_{F,W}$ for any $k$-graph $F$ and independent set $W\subseteq V (F)$.
When $W=\{w\}$ for a single vertex $w\in V(F)$, we omit the brackets and simply write $\Phi_{F,w}$ and $\Phi_{F\setminus w}$.
Below, we collect several fundamental properties of $\Phi_F$ and $\Phi_{F,W}$  that will be used in later sections. 

\begin{lemma}\label{lem:PropOfPhi}
The following hold:
  \begin{enumerate}[label=\rmlabel]
  \item \label{fact:phiK_{t+1}} 
  Let $t\geq k\geq 3$ be integers, and let $C >1$ be a constant. If $p=p(n)\geq C n^{-t/\binom{t+1}{k}}$, then $\Phi_{K_{t+1}^{(k)}}\geq C n$.
  \item \label{lem:F12}
  Let $F_1$ and $F_2$ be $k$-graphs with $V(F_1)\cap V(F_2)=\{v\}$, and let $W_i\subseteq V(F_i)$ be such that $v\notin W_1\cup W_2$ for $i=1,2$.
Then
\[\Phi_{F_1\cup F_2,W_1\cup W_2}\geq \min\{\Phi_{F_1, W_1},\Phi_{F_2, W_2},\Phi_{F_1, W_1}\Phi_{F_2, W_2} n^{-1}\}.\]
  \item \label{lem:phi_disjoint}
  Let $F_1$ and $F_2$ be $k$-graphs with $V(F_1)\cap V(F_2)=\emptyset$, and let $W_i\subseteq V(F_i)$ for $i=1,2$. 
  Suppose $\Phi_{F_1, W_1},\Phi_{F_2, W_2}\ge 1$.
  Then
\[\Phi_{F_1\cup F_2,W_1\cup W_2}\geq \min\{\Phi_{F_1, W_1},\Phi_{F_2, W_2}\}.\]
\end{enumerate}
\end{lemma}

\begin{proof}
For~\ref{fact:phiK_{t+1}}, let $F$ be a subgraph of $K_{t+1}^{(k)}$ with $v_F$ vertices and $e_F>0$ edges. 
Without loss of generality, we assume
that $F$ is an induced subgraph, so $e_F=\binom{v_F}{k}$.
Since $C > 1$ and $e_F \ge 1$, we observe that
$n^{v_F}p^{e_F}\ge C n^{v_F}\left( n^{-t /\binom{t+1}{k}}\right)^{e_F}$.
Therefore, to prove $\Phi_{K_{t+1}^{(k)}}\geq C n$, it suffices to show  $e_F\le \binom{t+1}{k}(v_F-1)/t$.
Using that $v_F\le t+1$ and the fact that the function $\frac{1}{x-1}\binom{x}{k}$ is increasing, we get $e_F= \binom{v_F}{k} \le \binom{t+1}{k}\frac{v_F-1}{t}$, 
as required.

Now we prove \ref{lem:F12} and \ref{lem:phi_disjoint}.
Let $F'$ be a subgraph of $F_1\cup F_2$ with $v'$ vertices and $e'$ edges, where $e'>0$.
  For $i\in\{1,2\}$, 
  write $v_i:=|V(F')\cap V(F_i)|$, $e_i:=|E(F')\cap E(F_i)|$, and $s_i\coloneqq|V(F')\cap W_i|$.
  Let $s\coloneqq|V(F')\cap (W_1\cup W_2)|$. Note that $s=s_1+s_2$.
  If $V(F')\subseteq V(F_i)$ for some $i\in\{1,2\}$, then
  \[
     n^{v'-s}p^{e'}=n^{v_i-s_i}p^{e_i}\geq \Phi_{F_i,W_i};
  \]
  otherwise, we have $V(F')\cap V(F_i)\neq \emptyset$ for $i\in\{1,2\}$.
For \ref{lem:F12}, since $v'\geq v_1 + v_2 -1$ and~$e'=e_1+e_2$, we have
  \[
    n^{v'-s}p^{e'}\geq
    n^{v_1+v_2-1-s}p^{e_1+e_2}=(n^{v_1-s_1}p^{e_1})(n^{v_2-s_2}p^{e_2})n^{-1}\geq
    \Phi_{F_1,W_1}\Phi_{F_2,W_2} n^{-1}.
  \]
This shows that $\Phi_{F_1\cup F_2,W_1\cup W_2}\geq \min\{\Phi_{F_1, W_1},\Phi_{F_2, W_2},\Phi_{F_1, W_1}\Phi_{F_2, W_2} n^{-1}\}$.
For \ref{lem:phi_disjoint},  since $v'=v_1 + v_2$ and~$e'=e_1+e_2$, we have
  \[
    n^{v'-s}p^{e'}\geq
    n^{v_1+v_2-s}p^{e_1+e_2}=(n^{v_1-s_1}p^{e_1})(n^{v_2-s_2}p^{e_2})\geq
    \Phi_{F_1,W_1}\Phi_{F_2,W_2}.
  \]
  This shows that $\Phi_{F_1\cup F_2,W_1\cup W_2}\geq \min\{\Phi_{F_1, W_1},\Phi_{F_2, W_2}\}$ since $\Phi_{F_1, W_1},\Phi_{F_2, W_2}\ge 1$.
\end{proof}


Now let us bound $\Phi_{F\setminus W}$ and $\Phi_{F,W}$ for some given $k$-graphs $F$ with an independent subset $W$ of $V(F)$ fixed in place, which will be used to prove our absorbing lemma.

The \emph{complete $k$-partite $k$-graph} with vertex classes $V_1,\dots,V_k$ is the $k$-graph on $V= V_1 \cup\cdots \cup V_k$ in which every set $e\subseteq V$ with $|e\cap V_i| = 1$ for each $1\le i\le k$ is an edge. 
We say that a complete $k$-partite $k$-graph is \emph{almost-balanced} if $\left||V_i|-|V_j|\right|\le 1$ for all $1\le i,j\le k$ with $i\neq j$.

\begin{definition}\label{def1}
Let $K$ be an almost-balanced complete $(k-1)$-partite $(k-1)$-graph on $t$ vertices and let $v,w\notin V(K)$  be distinct vertices.
Denote by 
$E_1=\left\{\{x\}\cup e:x\in\{w,v\},e\in 
E(\overline{K})\right\}$ and $E_2=\binom{V(K)}{k}$.
Let $A$ be a $k$-graph with vertex set $V(A)\coloneqq V(K)\cup \{w,v\}$ and edge set $E(A)\coloneqq E_1\cup E_2$.
\end{definition}

\begin{lemma}\label{l}
Let $t\geq k\ge 3$  be integers, and let $C>1$ be a
constant. Let
$0<\varepsilon<\min\left\{1,t/\binom{t+1}{k}\right\}$ and suppose $p=p(n)\ge Cn^{-t/\binom{t+1}{k}}$. Then the following hold.
\begin{itemize}
\item[$(i)$] Suppose in addition that Setup~\ref{setup1} holds. 
Let $A$ be the $k$-graph from Definition~\ref{def1}, and set $W \coloneqq \{w\}$. 
Then $\Phi_{A} \ge Cn^{1+\varepsilon}$, $\Phi_{A\setminus w} \ge Cn$, and $\Phi_{A,w} \ge Cn^{\varepsilon}$. 
\item[$(ii)$]  Let $K'$ be any non-complete $k$-graph on $t+1$ vertices with $W=\{w\}\subseteq V(K')$. 
Then 
$\Phi_{K'}\ge Cn^{1+\varepsilon}$, $\Phi_{K'\setminus W}\ge Cn$ and
$\Phi_{K',W}\ge Cn^\varepsilon$.
\end{itemize}
\end{lemma}

\begin{proof}
We first prove part~(i). Let $t\geq k \geq 3$ and $\varepsilon>0$ satisfy
    \begin{align}
        \binom{t}{k-1} \le 2 \left(\left\lceil\frac{t}{k-1}\right\rceil\right)^{\ell}\left(\left\lfloor\frac{t}{k-1}\right\rfloor\right)^{k-1-\ell}+\binom{t+1}{k}\frac{1-\varepsilon}{t},\label{equ:condition}
    \end{align}
    where $\ell \equiv t \pmod{k-1}$ is the unique integer such that $0 \le \ell \le k-2$.
    Note that the bounds for $\Phi_{A\setminus w}$ and $\Phi_{A,w}$ follow immediately from the stronger bound $\Phi_A \geq Cn^{1+\varepsilon}$. Thus, it suffices to prove $\Phi_A \geq Cn^{1+\varepsilon}$.
    
    Let $F$ be a subgraph of $A$ with $s$ vertices and $f > 0$ edges. 
    Without loss of generality, we assume that $F$ is an induced subgraph of $A$.
    Since $C>1$ and $f\ge 1$, we observe that $n^{s}p^{f} \ge Cn^s\left(n^{-t/\binom{t+1}{k}}\right)^f$. 
    Therefore, to prove $\Phi_A \ge Cn^{1+\varepsilon}$, it suffices to show $f \le \binom{t+1}{k}(s-1-\varepsilon)/t$. 
    Using the fact that $k\le s\le t+2$, we consider three cases for the size of $s$. 
 
\textbf{Case 1:} $s \le t$.
Because $\frac{1}{x-k+1}\binom{x}{k}$ is increasing, we have $\frac{1}{s-k+1}\binom{s}{k}\le \frac{1}{t-k+1}\binom{t}{k} = \frac{1}{t+1}\binom{t+1}{k}$.
Therefore, we have \[f \le \binom{s}{k}\le \binom{t+1}{k}\frac{s-k+1}{t+1} \le \binom{t+1}{k}\frac{s-2}{t}\] as $k\ge 3$.

\textbf{Case 2:} $s = t+1$.
Note that $F$ is a subgraph of  $A$ with $t+1$ vertices.
Since $A$ contains no copy of $K_{t+1}^{(k)}$, it follows that $f\le\binom{t+1}{k}-1$.
By $\varepsilon\le t/\binom{t+1}{k}$ and $s-1=t$,
we are done as 
\[f\le\binom{t+1}{k}-1\le \binom{t+1}{k}(s-1-\varepsilon)/t.\]

\textbf{Case 3:} $s = t+2$.
 It suffices to show $f\le \binom{t+1}{k} (t+1-\varepsilon)/t = \binom{t+1}{k}+\binom{t+1}{k}\frac{1-\varepsilon}{t}$.   
 We have
    \begin{align*}
        f &= 2\left(\binom{t}{k-1}-\left(\left\lceil\frac{t}{k-1}\right\rceil\right)^{\ell}\left(\left\lfloor\frac{t}{k-1}\right\rfloor\right)^{k-1-\ell}\right)+\binom{t}{k} \\
        &= \binom{t+1}{k}+\binom{t}{k-1}-2 \left(\left\lceil\frac{t}{k-1}\right\rceil\right)^{\ell}\left(\left\lfloor\frac{t}{k-1}\right\rfloor\right)^{k-1-\ell} \overset{\eqref{equ:condition}}{\le} \binom{t+1}{k}+\binom{t+1}{k}\frac{1-\varepsilon}{t},
    \end{align*}
as desired.

We now prove part~(ii). Let $K'$ be a non-complete $k$-graph on $t+1$ vertices. 
For every  subgraph $F\subseteq
K'$  with $s$ vertices and $f > 0$ edges,
Case 1 applies if $s\le t$. 
If $s=t+1$, then
$f\le\binom{t+1}{k}-1\le
\binom{t+1}{k}(s-1-\varepsilon)/t$ as in Case 2.
Therefore, we get  $\Phi_{K'}\ge Cn^{1+\varepsilon}$, and consequently
$\Phi_{K'\setminus W}\ge Cn$ and
$\Phi_{K',W}\ge Cn^\varepsilon$. 
\end{proof}

Next we present the definition of the $k$-graph $B$ and the corresponding estimates for $\Phi$, which are designed for Setup~\ref{setup2}.

\begin{definition}\label{def2}
Let $S$ be a set of $t$ vertices and let $v,w\notin S$  be distinct vertices.
Fix an almost-balanced complete $k$-partite
$k$-graph $S_v$ on $S\cup\{v\}$ and an almost-balanced complete
$k$-partite $k$-graph $S_w$ on $S\cup\{w\}$. We do not require
the vertex partitions of $S_v$ and $S_w$ to agree on $S$.
Define the $k$-graph $B=B(S;v,w;S_v,S_w)$ by
$V(B):=S\cup\{v,w\}$ and
\[E(B):=
\left(\binom{S\cup\{v\}}{k}
\cup
\binom{S\cup\{w\}}{k}\right)
\setminus
\left(E(S_v)\cup E(S_w)\right).\]
\end{definition}

 \begin{lemma}\label{l2}
Let $t\ge k\ge 3$ be integers, and let $C>1$ be a
constant. Let
$0<\varepsilon<\min\left\{1,t/\binom{t+1}{k}\right\}$ and suppose that Setup~\ref{setup2} holds with this value of $\varepsilon$.
For fixed $S_v$ and $S_w$, let $B$ be the $k$-graph defined in Definition~$\ref{def2}$, and set $W\coloneqq\{w\}$. 
If $p=p(n)\geq Cn^{-{t/\binom{t+1}{k}}}$, then $\Phi_{B}\ge Cn^{1+\varepsilon}$, $\Phi_{B\setminus w}\geq Cn$, and $\Phi_{B,w}\geq Cn^{\varepsilon}$. 
\end{lemma}

\begin{proof}
 Let $t\geq k\geq 3$  be integers satisfying 
\begin{align}
 \binom{t}{k-1} \le  \left(\left\lceil\frac{t+1}{k}\right\rceil\right)^{m}\left(\left\lfloor\frac{t+1}{k}\right\rfloor\right)^{k-m}+\left(\left\lceil\frac{t+1}{k}\right\rceil\right)^{m-1}\left(\left\lfloor\frac{t+1}{k}\right\rfloor\right)^{k-m}+\binom{t+1}{k}\frac{1-\varepsilon}{t},\label{equ:condition2}
\end{align}
where $m\equiv (t+1) \pmod{k}$ is the unique integer such that $0\le m\le k-1$.
Note that the bounds for $\Phi_{B\setminus w}$ and $\Phi_{B,w}$ follow immediately from the stronger bound $\Phi_B \geq Cn^{1+\varepsilon}$. Thus, it suffices to prove $\Phi_B \geq Cn^{1+\varepsilon}$.

Let $F$ be a subgraph of $B$ with $s$ vertices and $f > 0$ edges.
Without loss of generality, we assume that $F$ is an induced subgraph of $B$.
Since $C>1$ and $f\ge 1$, we observe that $n^{s}p^{f} \ge Cn^s\left(n^{-t/\binom{t+1}{k}}\right)^f$. 
Therefore, to prove $\Phi_B \ge Cn^{1+\varepsilon}$, it suffices to show $f \le \binom{t+1}{k}(s-1-\varepsilon)/t$. 
Using the fact that $k\le s\le t+2$, we consider three cases for the size of $s$. 

\textbf{Case 1:} $s \le t$.
Since $k\ge 3$, we have
$f\le \binom{s}{k}\le\binom{t+1}{k}\frac{s-2}{t}$ as in Case 1 of Lemma \ref{l},
which implies the required bound as $s-2 < s-1-\varepsilon$.

\textbf{Case 2:} $s= t+1$.
Note that $F$ is a subgraph of  $B$ with $t+1$ vertices. Since $B$ contains no copy of $K_{t+1}^{(k)}$, it follows that $f\le\binom{t+1}{k}-1$.
By $\varepsilon\le t/\binom{t+1}{k}$ and $s-1=t$,
we are done as $f\le\binom{t+1}{k}-1\le \binom{t+1}{k}(s-1-\varepsilon)/t$.



\textbf{Case 3:} $s= t+2$.
We have
\begin{align*}
 f
 &\leq\binom{t+1}{k}- \left(\left\lceil\frac{t+1}{k}\right\rceil\right)^{m}\left(\left\lfloor\frac{t+1}{k}\right\rfloor\right)^{k-m}+\binom{t}{k-1}-\left(\left\lceil\frac{t+1}{k}\right\rceil\right)^{m-1}\left(\left\lfloor\frac{t+1}{k}\right\rfloor\right)^{k-m}\\
 &\overset{\eqref{equ:condition2}}{\le} \binom{t+1}{k}+\binom{t+1}{k}\frac{1-\varepsilon}{t}= \binom{t+1}{k}\frac{s-1-\varepsilon}{t}.\nonumber 
\end{align*}
This completes the proof.
\end{proof}

\subsection{The absorbing structures}

Now we define the absorbing structure in this paper.

\begin{definition}[Comb]\label{def:com}
Let $s$ and $t$ be integers.
For a set $W=\{w_1,\dots,w_{t+1}\}$ of size $t+1$,
a comb $\mathcal{C}_W$ rooted at $W$ is a $k$-graph defined as follows. 
Let $\{P_{i,j}\}_{i\in[t+1],j\in[s]}$ be a collection of pairwise disjoint $t$-sets $(\text{i.e., } |P_{i,j}|=t \text{ for all } i\in[t+1],j\in[s])$ and let $\{u_{i,j}\}_{i\in[t+1],j\in[s]}$ be a collection of distinct vertices, such that all of these sets and vertices are pairwise disjoint and disjoint from $W$.
The vertex set of $\mathcal{C}_W$ is given by
\[
V(\mathcal{C}_W)=\left(\bigcup\nolimits_{i\in[t+1],j\in[s]}\left(P_{i,j}\cup \{u_{i,j}\}\right)\right)\cup W,
\]
and we set $u_{i,0}\coloneqq w_i$ for each $i\in[t+1]$.
The edge set of $\mathcal{C}_W$ is given by
\[
E(\mathcal{C}_W)=\left(\bigcup\nolimits_{i\in[t+1],j\in[s]} \big(E_{i,j}\cup E_{i,j}'\big)\right)\cup E^\ast,
\]
where $E_{i,j}$, $E_{i,j}'$ and $E^\ast$ denote the edge sets of the copies of the complete $k$-graph $K_{t+1}^{(k)}$ spanned by $\{u_{i,j-1}\}\cup P_{i,j}$, by $P_{i,j}\cup \{u_{i,j}\}$, and by $\{u_{1,s},\dots,u_{t+1,s}\}$, respectively.
Furthermore, we call \(s\) the width of \(\mathcal{C}_W\), \(W=\{w_1,\dots, w_{t+1}\}\) the base set of \(\mathcal{C}_W\), and \(\{u_{1,s},\dots,u_{t+1,s}\}\) the comb back of \(\mathcal{C}_W\). 
\textup{(See Figure~\ref{fig:comb} for an illustration of the case $k=3$, $s=3$ and $t=2$.)}
\end{definition}


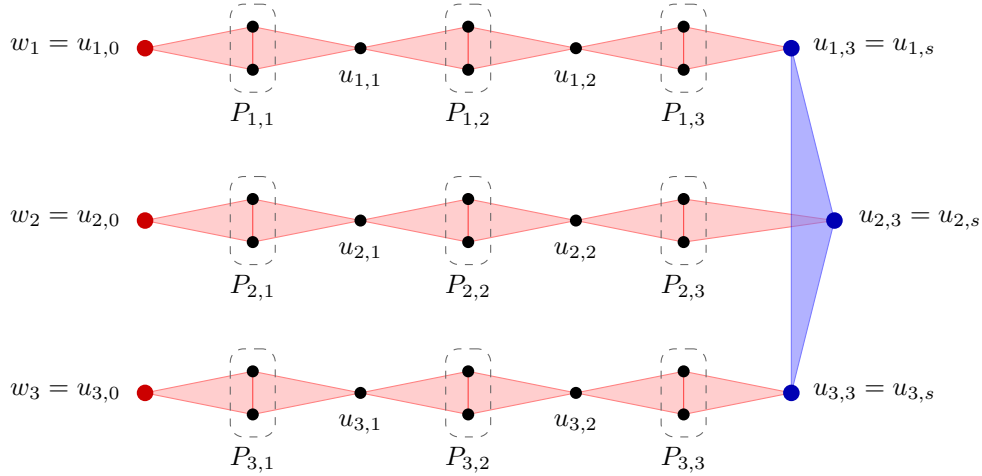
\begin{figure}[ht]
\centering
\begin{tikzpicture}[
    scale=0.95,
    vtx/.style={circle,fill=black,inner sep=1.6pt},
    root/.style={circle,fill=red!80!black,inner sep=2.2pt},
    tip/.style={circle,fill=blue!70!black,inner sep=2.2pt},
    he/.style={fill=red!35,draw=red!90,opacity=0.55,line join=round},
    sp/.style={fill=blue!55,draw=blue!95,opacity=0.55,line join=round},
    pbox/.style={draw=black!60,dashed,rounded corners=5pt,inner sep=6pt}
]
\def\dy{2.4}
\def\xw{0}
\def\xPa{1.5}\def\xua{3.0}
\def\xPb{4.5}\def\xub{6.0}
\def\xPc{7.5}

\def\xuI{9.0}    
\def\xuII{9.6}   
\def\xuIII{9.0}  

\foreach \i in {1,2,3}{
    \pgfmathsetmacro{\y}{(3-\i)*\dy}
    \node[root] (u\i0) at (\xw,\y) {};
    \node[vtx] (Pa\i x) at (\xPa,\y+0.30) {};
    \node[vtx] (Pa\i y) at (\xPa,\y-0.30) {};
    \node[vtx] (u\i1)   at (\xua,\y) {};
    \node[vtx] (Pb\i x) at (\xPb,\y+0.30) {};
    \node[vtx] (Pb\i y) at (\xPb,\y-0.30) {};
    \node[vtx] (u\i2)   at (\xub,\y) {};
    \node[vtx] (Pc\i x) at (\xPc,\y+0.30) {};
    \node[vtx] (Pc\i y) at (\xPc,\y-0.30) {};
}
\node[tip] (u13) at (\xuI,2*\dy)   {};   
\node[tip] (u23) at (\xuII,\dy)    {};   
\node[tip] (u33) at (\xuIII,0)     {};   

\begin{scope}[on background layer]
\foreach \i in {1,2,3}{
    \draw[he] (u\i0.center)--(Pa\i x.center)--(Pa\i y.center)--cycle;
    \draw[he] (Pa\i x.center)--(Pa\i y.center)--(u\i1.center)--cycle;
    \draw[he] (u\i1.center)--(Pb\i x.center)--(Pb\i y.center)--cycle;
    \draw[he] (Pb\i x.center)--(Pb\i y.center)--(u\i2.center)--cycle;
    \draw[he] (u\i2.center)--(Pc\i x.center)--(Pc\i y.center)--cycle;
    \draw[he] (Pc\i x.center)--(Pc\i y.center)--(u\i3.center)--cycle;
}
\draw[sp] (u13.center)--(u23.center)--(u33.center)--cycle;
\end{scope}

\foreach \i in {1,2,3}{
    \node[pbox,fit=(Pa\i x)(Pa\i y)] {};
    \node[pbox,fit=(Pb\i x)(Pb\i y)] {};
    \node[pbox,fit=(Pc\i x)(Pc\i y)] {};
}

\node[left=2pt of u10] {\small $w_1=u_{1,0}$};
\node[below=6pt of Pa1y] {\small $P_{1,1}$};
\node[below=2pt of u11]  {\small $u_{1,1}$};
\node[below=6pt of Pb1y] {\small $P_{1,2}$};
\node[below=2pt of u12]  {\small $u_{1,2}$};
\node[below=6pt of Pc1y] {\small $P_{1,3}$};
\node[right=1pt of u13] {\small $u_{1,3}=u_{1,s}$};

\node[left=2pt of u20] {\small $w_2=u_{2,0}$};
\node[below=6pt of Pa2y] {\small $P_{2,1}$};
\node[below=2pt of u21]  {\small $u_{2,1}$};
\node[below=6pt of Pb2y] {\small $P_{2,2}$};
\node[below=2pt of u22]  {\small $u_{2,2}$};
\node[below=6pt of Pc2y] {\small $P_{2,3}$};
\node[right=2pt of u23]  {\small $u_{2,3}=u_{2,s}$};

\node[left=2pt of u30] {\small $w_3=u_{3,0}$};
\node[below=6pt of Pa3y] {\small $P_{3,1}$};
\node[below=2pt of u31]  {\small $u_{3,1}$};
\node[below=6pt of Pb3y] {\small $P_{3,2}$};
\node[below=2pt of u32]  {\small $u_{3,2}$};
\node[below=6pt of Pc3y] {\small $P_{3,3}$};
\node[right=1pt of u33] {\small $u_{3,3}=u_{3,s}$};
\end{tikzpicture}
\caption{An illustration of the comb $\mathcal{C}_W$ with width $s=3$ rooted at 
$W=\{w_1,w_2,w_3\}$, shown for $k=3$ and $t=2$.
In this case $t+1=k=3$, so each copy of $K_{t+1}^{(k)}=K_3^{(3)}$ reduces to a
single $3$-edge (shaded triangles). Along each tooth, the
red triangles represent the edges $E_{i,j}$ and $E_{i,j}'$, which are spanned by
$\{u_{i,j-1}\}\cup P_{i,j}$ and $P_{i,j}\cup\{u_{i,j}\}$, respectively. The blue triangle at the top represents the edge $E^\ast$ spanned by the comb back
$\{u_{1,s},u_{2,s},u_{3,s}\}$.}
\label{fig:comb}
\end{figure}

Let $p\ge Cn^{-t/\binom{t+1}{k}}$ and let $H$ be an $n$-vertex $k$-graph with $\delta_{k-1}(H)\ge \alpha n$. To build vertex-disjoint copies of $K_{t+1}^{(k)}$, we may have to use both 
deterministic edges from $H$ and random edges from $H^{(k)}(n,p)$.
To find such vertex-disjoint copies, we will use the following partition of the edge set of the comb $\mathcal{C}_W$ rooted at $W$.

\begin{definition}[$H_{\text{det}}^1$ and $H_{\text{ran}}^1$]\label{def:H_{det}^1}
Let $\mathcal{C}_W$ be a comb rooted at a $(t+1)$-set $W$ with width $s$ as defined in Definition~\ref{def:com}, and let $K_{i,j}^{(k-1)}(t)$ be the complete $(k-1)$-partite $(k-1)$-graph on $t$ vertices with vertex classes $U_{i,j}^1\cup \cdots\cup U_{i,j}^{k-1}=P_{i,j}$ satisfying $\left||U_{i,j}^x|-|U_{i,j}^y|\right|\le 1$ for all $x,y\in [k-1]$. 
Then, we define $H_{\text{det}}^1=H_{\text{det}}^1(k,s,t)$ as the subgraph of $\mathcal{C}_W$ on the same vertex set with edge set $E\left(H_{\text{det}}^1\right)\coloneqq E_1\cup E_2$, where 
\[
E_1\coloneqq\left\{\{u_{i,j}\}\cup e\,\colon\,e\in E\left(K_{i,j+1}^{(k-1)}(t)\right)\right\} \mbox{ for } i\in[t+1],j\in[s-1]\cup\{0\},
\]
with $u_{i,0}\coloneqq w_i$; and
\[
E_2\coloneqq\left\{e\cup \{u_{i,j}\}\,\colon\,e\in E\left(K_{i,j}^{(k-1)}(t)\right)\right\} \mbox{ for } i\in[t+1],j\in[s].
\]
We further define $H_{\text{ran}}^1=H_{\text{ran}}^1(k,s,t)$ as the subgraph of $\mathcal{C}_W$ on the same vertex set such that 
\[E(H_{\text{ran}}^1)\coloneqq E\left(\mathcal{C}_W\right)\setminus E\left(H_{\text{det}}^1\right).
\]
\end{definition}

Now let us bound $\Phi_{H_{\text{ran}}^1\setminus W}$ and $\Phi_{H_{\text{ran}}^1,W}$ with an independent subset $W\subseteq V(H_{\text{ran}}^1)$ fixed in place, which will be used to prove our absorbing lemma.
For its proof, we use Lemma~\ref{l} to control the building blocks of $H_{\text{ran}}^1$, and concatenate the blocks by Lemma~\ref{lem:PropOfPhi}.

\begin{lemma}\label{Phi:abs1}
Suppose that Setup~\ref{setup1} holds and let $C >1$ be a constant.
Let $H_{\text{ran}}^1=H_{\text{ran}}^1(k,s,t)$ be the $k$-graph as defined in Definition~$\ref{def:H_{det}^1}$ with $s=\lceil 1/\varepsilon\rceil$ and $W=\{w_1,\dots,w_{t+1}\}$. 
If $p=p(n)\geq Cn^{-{t/\binom{t+1}{k}}}$, then $\Phi_{H_{\text{ran}}^1\setminus W}\ge Cn$ and $\Phi_{H_{\text{ran}}^1,W}\ge Cn^{\varepsilon}$.
\end{lemma}

\begin{proof}
Let $\mathcal{C}_W$ be a comb rooted at $W=\{w_1,\dots,w_{t+1}\}$ with width $s=\lceil 1/\varepsilon\rceil$, where the comb back is $U=\{u_{1,s},\dots,u_{t+1,s}\}$.
Let $H_{\text{ran}}^1=H_{\text{ran}}^1(k,s,t)$ be the subgraph of $\mathcal{C}_W$ defined as in Definition~$\ref{def:H_{det}^1}$. 
For each $i\in[t+1]$, let
\[V_i\coloneqq\bigcup_{j\in[s]}\left(P_{i,j}\cup \{u_{i,j}\}\right)\cup \{w_i\},\quad H_i\coloneqq H_{\text{ran}}^1[V_i],\] and
\[H_{i,j}\coloneqq H_{\text{ran}}^1[\{u_{i,j-1}\}\cup P_{i,j}\cup \{u_{i,j}\}]\]
for $j\in[s]$, where $u_{i,0}\coloneqq w_i$.
Then each $H_{i,j}$ is isomorphic to the $k$-graph $A$ from Definition~\ref{def1}. By Lemma~\ref{l}, we have 
\[\Phi_{H_{i,j}}\ge Cn^{1+\varepsilon}, \quad \Phi_{H_{i,1}\setminus w_i}\ge Cn, \quad \Phi_{H_{i,1}, w_i}\ge Cn^{\varepsilon}.\] 
Moreover, $H_{\text{ran}}^1$ consists of a copy $H_{0}$ of $K_{t+1}^{(k)}$ on $U$, together with the $(t+1)s$ copies $H_{i,j}$ of $A$.

We first prove that $\Phi_{H_{\text{ran}}^1\setminus W}\ge Cn$.
Starting with $H_0$, we can construct $H_{\text{ran}}^1\setminus W$ by sequentially gluing, for each $i\in[t+1]$, the subgraphs
\[H_{i,s},H_{i,s-1},\dots,H_{i,2},H_{i,1}\setminus w_i,\]
where each newly added subgraph intersects the previously constructed subgraph in exactly one vertex.
By Lemma~\ref{lem:PropOfPhi}~\ref{fact:phiK_{t+1}}, we have $\Phi_{H_0}= \Phi_{K_{t+1}^{(k)}}\ge Cn$.
Combining this with the bounds $\Phi_{H_{i,1}\setminus w_i}\ge Cn$ and $\Phi_{H_{i,j}}\ge Cn^{1+\varepsilon}$, we can repeatedly apply Lemma~\ref{lem:PropOfPhi}~\ref{lem:F12} to obtain $\Phi_{H_{\text{ran}}^1\setminus W}\ge Cn$.

Next, we prove that $\Phi_{H_{\text{ran}}^1,W}\ge Cn^\varepsilon$.
Let $F$ be a subgraph of $H_{\text{ran}}^1$ with $v_F$ vertices and $e_F>0$ edges.
It suffices to show that $n^{v_F-|V(F)\cap W|}p^{e_F}\geq Cn^{\varepsilon}$.
By considering the connected components of $F$ and applying Lemma~\ref{lem:PropOfPhi}\ref{lem:phi_disjoint}, it suffices to consider the case where $F$ is connected.
We distinguish two cases according to whether $F$ meets the comb back $U$.

\textbf{Case 1:} $V(F)\cap U=\emptyset$.
Then $F\subseteq H_i$ for some $i\in[t+1]$, say $i=1$. If $w_1\notin V(F)$, then
$F\subseteq H_{\text{ran}}^1\setminus W$, and hence
\[n^{v_F-|V(F)\cap W|}p^{e_F}=n^{v_F}p^{e_F}
\ge \Phi_{H_{\text{ran}}^1\setminus W}
\ge Cn
\ge Cn^\varepsilon.\]
Otherwise, we have $w_1\in V(F)$.
Since $F$ is connected, it can be obtained by sequentially gluing the nonempty
subgraphs $F\cap H_{1,j}$ along the connecting vertices in
$\{u_{1,1},\dots,u_{1,s-1}\}$.
Using
\[
\Phi_{H_{1,1},w_1}\ge Cn^\varepsilon
\quad\text{and}\quad
\Phi_{H_{1,j}}\ge Cn^{1+\varepsilon}
\quad (j\in[2,s]),
\]
repeated applications of
Lemma~\ref{lem:PropOfPhi}\ref{lem:F12} yield
\[
n^{v_F-|V(F)\cap W|}p^{e_F}=n^{v_F-1}p^{e_F}\ge Cn^\varepsilon.
\]

\textbf{Case 2:} $V(F)\cap U\neq\emptyset$. 
For each $i\in[t+1]$, let $F_i:=F[V(H_i)\cap V(F)]$,
and define $\mathcal{I}\coloneqq\{i\in[t+1]:u_{i,s}\in V(F)\}$.
We claim that, for every $i\in\mathcal I$ with $e(F_i)>0$,
\[
n^{v_{F_i}-|V(F_i)\cap W|}p^{e_{F_i}}\ge Cn.
\]
Indeed, if $w_i\notin V(F_i)$, then \[
n^{v_{F_i}-|V(F_i)\cap W|}p^{e_{F_i}}
=
n^{v_{F_i}}p^{e_{F_i}}
\ge
\Phi_{H_{\mathrm{ran}}^1\setminus W}
\ge Cn.
\]
Suppose now that $w_i\in V(F_i)$.
Since $F$ is connected and $w_i,u_{i,s}\in V(F_i)$, the subgraph $F_i$ connects $w_i$ to $u_{i,s}$ and hence contains all vertices $u_{i,1},\dots,u_{i,s-1}$.
Thus, for every $j\in[s]$, the subgraph $F_{i,j}:=F_i\cap H_{i,j}$ contains at least one edge. Since consecutive $F_{i,j}$ intersect in exactly one vertex, we have
\[
\begin{aligned}
n^{v_{F_i}-1}p^{e_{F_i}}
&=
\left(n^{v_{F_{i,1}}-1}p^{e_{F_{i,1}}}\right)
\prod_{j=2}^s
\left(n^{v_{F_{i,j}}}p^{e_{F_{i,j}}}\right)
n^{-(s-1)}
\\
&\ge
\frac{Cn^\varepsilon
\left(Cn^{1+\varepsilon}\right)^{s-1}}
{n^{s-1}}=
C^s n^{s\varepsilon}
\ge Cn,
\end{aligned}
\]
where the last inequality follows from $C>1$ and
$s=\lceil1/\varepsilon\rceil$.
This proves the claim.

Finally, let $F_0:=F\cap H_0$ and $\mathcal I_F:=\{i\in\mathcal I:e(F_i)>0\}$.
If $e(F_0)=0$, then the connectivity of $F$ implies that all edges of
$F$ are contained in some $H_i$. Hence, the above claim gives
$n^{v_F-|V(F)\cap W|}p^{e_F}\ge Cn$.
Suppose now that $e(F_0)>0$. Since
$n^{v_{F_0}}p^{e_{F_0}}
\ge \Phi_{H_0}\ge Cn$,
and $F$ is obtained by gluing $F_0$ and the subgraphs $F_i$ with
$i\in\mathcal I_F$ along vertices in $U$, we obtain
\[
\begin{aligned}
n^{v_F-|V(F)\cap W|}p^{e_F}
&=
\left(n^{v_{F_0}}p^{e_{F_0}}\right)
\prod_{i\in\mathcal I_F}
\left(
n^{v_{F_i}-|V(F_i)\cap W|}p^{e_{F_i}}
\right)
n^{-|\mathcal I_F|}
\\
&\ge\left(Cn\right)^{|\mathcal I_F|+1}
n^{-|\mathcal I_F|}
\ge Cn .
\end{aligned}
\]

Thus, in either case,
$n^{v_F-|V(F)\cap W|}p^{e_F}\ge Cn^\varepsilon$.
Hence,
$\Phi_{H_{\text{ran}}^1,W}\ge Cn^\varepsilon$,
which completes the proof.
\end{proof}

Similarly, we present the second construction of the absorbers, working under Setup \ref{setup2}.

\begin{definition}[$H_{\text{det}}^2$ and $H_{\text{ran}}^2$]\label{def:H_{det}^2}
Let $\mathcal{C}_W$ be a comb rooted at $W$ with width $s$ as defined in Definition~\ref{def:com}. 
For each $i\in [t+1]$ and $j\in [s]$, define the left attachment $L_{i,j}$ as the almost-balanced $k$-partite $k$-graph on $\{u_{i,j-1}\}\cup P_{i,j}$, where $u_{i,0}\coloneqq w_i$; and the right attachment $R_{i,j}$ as the almost-balanced $k$-partite $k$-graph on $P_{i,j}\cup \{u_{i,j}\}$.
Then, we define $H_{\text{det}}^2=H_{\text{det}}^2(k,s,t)$ as the subgraph of $\mathcal{C}_W$ on the same vertex set with edge set 
\[E\left(H_{\text{det}}^2\right)\coloneqq \bigcup_{i\in[t+1],j\in[s]}\left(E(L_{i,j})\cup E(R_{i,j})\right).\] 
We further define $H_{\text{ran}}^2=H_{\text{ran}}^2(k,s,t)$ as the subgraph of $\mathcal{C}_W$ on the same vertex set such that 
\[E(H_{\text{ran}}^2)\coloneqq E\left(\mathcal{C}_W\right)\setminus E\left(H_{\text{det}}^2\right).
\]
\end{definition}

We now bound $\Phi_{H_{\text{ran}}^2\setminus W}$ and $\Phi_{H_{\text{ran}}^2,W}$ with a fixed independent subset $W\subseteq V(H_{\text{ran}}^2)$. This will be used to prove our absorbing lemma.
Its proof is almost identical to the one for Lemma \ref{Phi:abs1}. 

\begin{lemma}\label{Phi:abs2}
Suppose that Setup~\ref{setup2} holds and let $C >1$ be a constant.
Let $H_{\text{ran}}^2=H_{\text{ran}}^2(k,s,t)$ be the $k$-graph as defined in Definition~$\ref{def:H_{det}^2}$ with $s=\lceil 1/\varepsilon\rceil$ and $W=\{w_1,\dots,w_{t+1}\}$. 
If $p=p(n)\geq Cn^{-{t/\binom{t+1}{k}}}$, then $\Phi_{H_{\text{ran}}^2\setminus W}\ge Cn$ and $\Phi_{H_{\text{ran}}^2,W}\ge Cn^{\varepsilon}$.
\end{lemma}

\begin{proof}
The proof is identical to that of Lemma~\ref{Phi:abs1} after replacing the 
$k$-graph 
$A$ from Definition~\ref{def1} with 
$B$ from Definition~\ref{def2} as well as the use of Lemma~\ref{l} with Lemma~\ref{l2}.
For each $i\in[t+1]$ and $j\in[s]$, the induced subgraph $H_{\text{ran}}^2[\{u_{i,j-1}\}\cup P_{i,j}\cup\{u_{i,j}\}]$,
where $u_{i,0}\coloneqq w_i$, is isomorphic to a $k$-graph $B$ as
defined in Definition~\ref{def2}.
Thus, Lemma~\ref{l2} gives the same local bounds as Lemma~\ref{l} in the proof of Lemma~\ref{Phi:abs1}.
Since the global gluing structure of $H_{\text{ran}}^2$ is identical to that of $H_{\text{ran}}^1$, the remainder of the proof follows from the same decomposition and gluing arguments. We omit the repeated details here.
\end{proof}

\section{The absorbing Lemma}
\label{s4}
In this section, we present an absorbing lemma (Lemma~\ref{lem:absorbing}) proved in \cite{han2025perturbation}, which in turn builds on other previous works \cite{chang2022factors,HMT2021,M,nenadov2020ramsey}. 
Then the main work of this section, in Lemma \ref{lem:absa} and Lemma \ref{vertexcover}, is to verify the hypotheses in Lemma~\ref{lem:absorbing}. 


\begin{definition}
  \label{def:abs}
  \emph{Let $H$ be an $n$-vertex $k$-graph, and let $S\subseteq V(H)$ be a subset of size $t+1$. A subset $A \subseteq V(H)\setminus S$ is called an \emph{$S$-absorber} if both $H[A]$ and $H[A \cup S]$ contain $K_{t+1}^{(k)}$-factors.}
\end{definition}


The following lemma guarantees the existence of an absorbing set capable of absorbing any sufficiently small subset of vertices, provided the (robust) existence of absorbers.
We remark that, in the statement below, by saying that $H^*$ is
a random $k$-graph, we mean that $H^*$ is drawn from certain probability distribution (so not necessarily
$H^{(k)}(n,p)$).

\begin{lemma}[Absorbing Lemma, \cite{han2025perturbation}]\label{lem:absorbing}
Let $F$ be a  $k$-graph with $m$ vertices. Suppose $1/n\ll \xi \ll \gamma_1\ll \tau \ll \gamma \ll 1/m$.  
   For any $n$-vertex random $k$-graph $H^*$, suppose there is a constant $C_1>0$ such that  
   $H^*$  satisfies the following  properties:
\begin{itemize}
\item[$(i)$] There is a vertex set $V_1\subseteq V(H^*)$ with $\tau n/4 \le|V_1|\le \tau n$ such that for any $S\subseteq V(H^*)$ with $|S|\le \gamma_1 n$, a.a.s.~there exists a set $T\subseteq V_1\setminus S$ with $|T|=(m-1)|S|$ and $H^*[S\cup T]$ containing an $F$-factor. 
\item[$(ii)$]  Let $q\le \frac{120\tau}{\gamma_1}$ be a constant and  $\mathcal{S}_1, \dots,\mathcal{S}_q$ be  families of $m$-sets of  $V(H^*)$ such that each $\mathcal{S}_i$ is a family of pairwise disjoint sets. 
Let $\mathcal{S}_i=\{S_{i,1},\dots,S_{i,t_i}\}$, where $t_i=|\mathcal{S}_i|\le \gamma_1 n$. Then a.a.s.~there exist vertex-disjoint families of $C_1$-sets $\left\{T_{i,j}\,\colon\,  i\in [q] , j\in [t_i]\right\}$ in $V(H^*)\setminus \bigcup_{i\in [q], j\in [t_i]} S_{i,j}$  
 such that for any $i\in [q]$ and $j\in [t_i]$, $ T_{i,j}$ is an $S_{i,j}$-absorber. 
\end{itemize} 
   Then $H^*$ contains an absorbing set $V_A\subseteq V(H^*)$ of size at most $\gamma n$ such that, for every $R\subseteq V(H^*)\setminus V_A$ with $|R|\le \xi n$ such that $m$ divides $|V_A|+|R|$, a.a.s.~$H^*[V_A\cup R]$ contains an $F$-factor.   
\end{lemma} 

Therefore, it suffices to establish the assumptions in Lemma \ref{lem:absorbing} which then will produce an absorbing set.
To prepare for the construction of an absorbing set, the main tool is the following proposition from~\cite{CHL2023}, adapted from~\cite[Proposition 6.21]{HMT2021}, which provides a multi-round greedy embedding scheme for generating an absorbing set.

\begin{proposition}\cite[Proposition 5.2]{CHL2023} \label{p1}
Let $k,b,f,w,\ell,n,t=t(n)\in \mathbb{N}$ be such that $k\ge 2$, and suppose $1/n\ll 1/C,\gamma\ll \beta,1/f,1/b,1/w,1/k,1/\ell$.
Let $F$ be a labelled $k$-graph with base vertex set $W\subseteq V(F)$ such that $|W| =w$, $|V (F)\setminus W|=b$, $e(F[W]) = 0$ and $e(F)= f$.
Further, suppose that $p = p(n)$ is such that $\Phi_{F\setminus W} \geq Cn$ and $\Phi_{F, W}\geq Cn^{1/\ell}$.

Let $V$ be an $n$-vertex set, and let $U_1,\dots, U_t\subseteq V$ with $t\le \gamma n$ be pairwise disjoint and be such that $|U_i| = |W|$ for each $i\in[t]$.
Suppose that $\mathcal{F}_1,\dots, \mathcal{F}_t$ are families of ordered $b$-sets on $V$ such that $|\mathcal{F}_i|\geq \beta n^b$.
Then a.a.s.~there exists a collection of embeddings $\phi_1, \dots, \phi_t$ such that each $\phi_i$ embeds a copy of $F$ into $H^{(k)}(n, p)$ on $V$ with $W$ being mapped to $U_i$ and $V(F)\setminus W$ being mapped to a set in $\mathcal{F}_i$ which is vertex-disjoint with $\bigcup_{i\in[t]}U_i$.
Furthermore, for $i\neq j$ we have $\phi_i(V (F)\setminus W) \cap \phi_j(V (F) \setminus W) = \emptyset$.
\end{proposition}

Let $K_{k}^{k}(m)$ denote the almost-balanced complete $k$-partite $k$-graph with $m$ vertices. Our proof of Lemma~\ref{lem:absa} is based on the following form of a lemma from \cite{lo2015f}. 

\begin{lemma}\cite[Lemma 4.2]{lo2015f}
\label{lemma:3.2}
Let $k,m \geq 2$ be integers and let $\eta>0$. Then there exist a constant $\beta_0 = \beta_0(k,m,\eta) > 0$ and an integer $n_0 = n_0(k,m,\eta)$ satisfying the following property. Let $H$ be a $k$-graph on $n \geq n_0$  vertices with $\delta_{k-1}(H)\ge \eta n$. 
If $|N(u)\cap N(v) |\ge \eta \binom{n}{k-1}$ for any two vertices $u,v \in V(H)$, then there are at least $\beta_0n^{m-1}$ $(m-1)$-sets $W$ such that $W\cap\{u,v\}=\emptyset$, and both $H[\{u\} \cup W]$ and $H[\{v\} \cup W]$ contain $K_{k}^{k}(m)$. 
\end{lemma}


Indeed, the lemma above was proved in \cite[Lemma 4.2]{lo2015f} for balanced complete $k$-partite $k$-graphs; the current form we need follows from simple averaging.

\subsection{Verification of condition~$(ii)$  of Lemma~\ref{lem:absorbing}}
We prove in Lemma \ref{lem:absa} that under appropriate conditions on $p$ and the minimum codegree of $H$, the union $H \cup H^{(k)}(n,p)$ a.a.s.~contains many $S$-absorbers for every $(t+1)$-subset $S \subseteq V(H)$. 

For a vertex $v\in V(H)$, we define $\tilde{N}_{\eta}(v)$ as the set of all vertices $u\in V(H)\setminus \{v\}$  such that  $|N_H(u) \cap N_H(v) |\ge \eta  n^{k-1}$.
The following result provides a lower bound on the size of $\tilde{N}_{\eta}(v)$ in $H$.
\begin{proposition} \label{prop:V0}
   Let $k\ge 3$ and suppose $1/n\ll\eta\ll\alpha,1/k$.
Let $H=(V,E)$ be an $n$-vertex $k$-graph with $\delta_{k-1}(H)\ge \alpha n$.
Then, for any $v \in V$, $|\tilde{N}_{\eta}(v)| \geq \alpha n/2$.
\end{proposition}
\begin{proof}
Since $\delta_{k-1}(H)\ge \alpha n$, it follows that $\delta_{1}(H)\ge \frac{1}{k-1}\binom{n-1}{k-2}\alpha n =\alpha\binom{n}{k-1}$. 
Fix a vertex $v \in V$.
For any vertex $u$, by definition $u \in \tilde{N}_{\eta}(v)$ if and only if $|N_H(u) \cap N_H(v)| \geq \eta  n^{k-1}$. By a double-counting argument, we have
    \[
    |N_H(v)|\cdot \delta_{k-1}(H) \leq \sum_{S \in N_H(v)} \deg_H(S) < |\tilde{N}_{\eta }(v)| \cdot |N_H(v)| + n \cdot \eta  n^{k-1}.
    \]
Since $|N_H(v)| \geq \delta_1(H) \geq \alpha\binom{n}{k-1}$ and $\eta\ll\alpha$, we conclude that
    \[
    |\tilde{N}_{\eta}(v)| > \delta_{k-1}(H) - \frac{\eta n^k}{|N_H(v)|} \geq \alpha n/2. \qedhere
    \]
\end{proof}
We are now ready to establish condition~$(ii)$  of Lemma~\ref{lem:absorbing} in our setting.

\begin{lemma}
  \label{lem:absa}
Given Setup~\ref{setup1} or Setup~\ref{setup2}, suppose $1/n\ll\gamma_1\ll\tau\ll\beta\ll\varepsilon\ll \alpha,1/k,1/t$.
Let $H=(V,E)$ be an $n$-vertex $k$-graph with $\delta_{k-1}(H)\ge \alpha n$.
  Then there is a constant $C = C(\alpha,\varepsilon) > 0$ such that if $p\geq Cn^{-t/\binom{t+1}{k}}$, then $H^{*}:=H\cup H^{(k)}(n,p)$ satisfies condition $(ii)$ in Lemma~\ref{lem:absorbing} with $C_1=\left\lceil 1/\varepsilon\right\rceil(t+1)^2$ for $F=K_{t+1}^{(k)}$.
\end{lemma}
\begin{proof}
Let $H$  be defined as in the statement of the lemma. We define additional constants
such that
\[1/n\ll\gamma_1\ll\tau\ll\beta\ll\beta'
\ll\gamma\ll\eta
\ll\varepsilon \ll \alpha,1/k,1/t.\]
Let $q\le \frac{120\tau}{\gamma_1}$ and consider such families of $(t+1)$-sets $\mathcal{S}_1, \dots,\mathcal{S}_q$ of $V$, where $\mathcal{S}_i=\{S_{i,1},\dots,S_{i,t_i}\}$ and $t_i=|\mathcal{S}_i|\le \gamma_1 n$. 
Let $s=\left\lceil 1/\varepsilon\right\rceil$, $C_1=s(t+1)^2$, and $\ell=\big\lceil\frac{120\tau}{\gamma_1}\big\rceil$. 
Let $C>0$ be a constant such that the conclusions of both Lemma~\ref{Phi:abs1} and Lemma~\ref{Phi:abs2} hold with $C/\ell$ in place of $C$.

Let $p\geq Cn^{-t/\binom{t+1}{k}}$. 
We will show that a.a.s.~there exist vertex-disjoint families of $C_1$-sets $\{T_{i,j}\,\colon\,  i\in [q] , j\in [t_i]\}$ in $V\setminus \bigcup_{i\in [q], j\in [t_i]} S_{i,j}$ such that, for every $i\in [q]$ and $j\in [t_i]$, $ T_{i,j}$ is an $S_{i,j}$-absorber.
We expose $H'\coloneqq H^{(k)}(n, p)$ in $\ell$ rounds: $H'= \bigcup_{i=1}^{\ell} H_i$, where each $H_i$ is an independent copy of $H^{(k)}(n, p')$ satisfying $(1 - p')^{\ell} = 1 - p$. Note that, since $(1 - p')^{\ell} \geq 1 - \ell p'$, it follows that $p' \geq p / \ell\geq  C'n^{-t/\binom{t+1}{k}}$, where $C'=C/\ell$. 
We now establish the following claim.

\begin{claim}
  \label{S}
Let $X \subseteq V$ with $|X| \ge n-\eta^2 n$. 
Let $\nu\le \gamma_1 n$, and let $\{S_1, \dots, S_\nu\}$ be a family of vertex-disjoint $(t+1)$-subsets of  $V$. 
Then, for each $j\in [\ell]$, a.a.s.~$H\cup H_j$ contains vertex-disjoint $s(t+1)^2$-subsets $T_1, \dots, T_\nu$ in $X \setminus \left( \bigcup_{i\in [\nu]} S_i \right)$ such that each $T_i$ is an $S_i$-absorber.
\end{claim}
\begin{proof}
[Proof of the claim]
Fix $j\in [\ell]$. Our objective is to find $\nu$ vertex-disjoint $s(t+1)^2$-sets $T_1,\dots,T_{\nu}$ in $X \setminus \left( \bigcup_{i\in [\nu]} S_i \right)$ such that, for every $i\in[\nu]$, $S_i\cup T_i$ spans a copy of the comb $\mathcal{C}_{W}$ rooted at $W$ with $S_i$ playing the role of $W$ in $H\cup H_j$ (see Definition~\ref{def:com}).

For Setup~\ref{setup1}, let $S=\{s_{1},\dots,s_{t+1}\}$ be a $(t+1)$-set in $V$. 
To build $S$-absorbers, we first build certain structures using deterministic edges and then supply the missing edges using random edges.
The first step is to find the candidates for $S$-absorbers.
Specifically, we find ordered $s(t+1)^2$-sets $Y\coloneqq\bigcup_{i\in[t+1],j\in[s]}\left(P_{i,j}\cup\{u_{i,j}\}\right)$ such that $|P_{i,j}|=t$ and $Y\cup S$ spans a copy of 
$H_{\text{det}}^1$ with $S$ playing the role of $W$ (see Definition \ref{def:H_{det}^1}).
By Proposition~\ref{prop:V0}, for any $v \in V$, $|\tilde{N}_{\eta}(v)| \geq \alpha n/2$.
So there are at least $\alpha n/2-s(t+1)^2-\nu(t+1)-\eta^2 n\ge \alpha n/4$ choices for $u_{1,1}\in X$ such that $|N_{H}(s_1)\cap N_{H}(u_{1,1})|\ge \eta n^{k-1}$ and $u_{1,1}$ is disjoint from both $\bigcup_{i\in [\nu]} S_i$ and all previously chosen vertices. 
Let $K_{i,j}^{(k-1)}(t)$ denote the complete $(k-1)$-partite $(k-1)$-graph on the vertex set $P_{i,j}$ for $i\in[t+1]$ and $j\in[s]$. 
Recall that the link $L_v$ of a vertex $v\in V(H)$ is the $(k-1)$-graph with vertex set $V(L_v)=V(H)\setminus \{v\}$ and edge set $E(L_v)=\{S\colon\, S\cup \{v\}\in E(H)\}$. Then 
$|E(L_{s_1}\cap L_{u_{1,1}})|\ge\eta n^{k-1}$ and $|E(L_{s_1}[X]\cap L_{u_{1,1}}[X])|\ge\eta n^{k-1}-\eta^2 n^{k-1}\ge \eta n^{k-1}/2 $.
By the supersaturation result (see~\cite[Corollary 2]{Erdos}), there are at least $\beta'n^t$ copies of $K_{1,1}^{(k-1)}$ in $L_{s_1}[X]\cap L_{u_{1,1}}[X]$.
Thus, there are at least $\beta'n^t-[s(t+1)^2+\nu(t+1)]n^{t-1}\ge \beta'n^t/2$ choices for $P_{1,1}$ such that, for every $e\in E\left(K_{1,1}^{(k-1)}\right)$, both $\{s_1\}\cup e$ and $e\cup \{u_{1,1}\}$ are edges in $H$.
Similarly, we greedily choose  $u_{1,2}, P_{1,2}, \dots, u_{1,s},P_{1,s}$ such that each of them is vertex-disjoint from both $\bigcup_{i\in [\nu]} S_i$ and all previously chosen vertices. 
Repeating this process for all $i\in[t+1]$, there are at least 
\[\left(\alpha n/4\right)^{s(t+1)}\left(\beta' n^{t}/2\right)^{s(t+1)}\ge \beta n^{s(t+1)^2}\]
desired ordered $s(t+1)^2$-sets $Y$ for $S$ such that $Y\cup S$ spans a copy of 
$H_{\text{det}}^1$.
For every $i\in[\nu]$, let $\mathcal{F}_i$ denote the family of such ordered $s(t+1)^2$-sets corresponding to the $(t+1)$-set $S_i$. Then, $|\mathcal{F}_i|\ge \beta n^{s(t+1)^2}$.

The second step is to find an $S_i$-absorber from $\mathcal{F}_i$ for every $i\in[\nu]$ by adding edges from $H_j$ to fill the missing
edges.
Let $H_{\text{ran}}^1$ and $W$ be the labeled $k$-graph and vertex set  defined in Definition~\ref{def:H_{det}^1}.
Observe that $H_{\text{ran}}^1$ is isomorphic to the $k$-graph induced by the required missing edges.
By Lemma~\ref{Phi:abs1}, we have 
$\Phi_{H_{\text{ran}}^1\setminus W}(n,p')\geq C'n$ and $\Phi_{H_{\text{ran}}^1, W}(n,p')\geq C'n^{\varepsilon}$. 
To apply Proposition~\ref{p1}, we set $F=H_{\text{ran}}^1$, consider the set $W$, and use the families $\mathcal{F}_{i}$ for each $i\in[\nu]$. 
We set the parameters $b=s(t+1)^2$ and $w=t+1$, and take $U_i=S_{i}$ for each $i\in[\nu]$. 
This guarantees that, a.a.s., there exist pairwise disjoint embeddings $\phi_1,\dots,\phi_{\nu}$. 
Each $\phi_{i}$ embeds a copy of $H_{\text{ran}}^1$ into $H_j$, such that $W$ is mapped to $S_{i}$ and $V(H_{\text{ran}}^1)\setminus W$ is mapped to some $s(t+1)^2$-set in $\mathcal{F}_{i}$. 
Consequently, we obtain $\nu$ vertex-disjoint $s(t+1)^2$-sets $T_1, \dots, T_{\nu}$ in $X\setminus \bigcup_{i\in [\nu]} S_{i}$ such that each $T_i$ is an $S_i$-absorber in $H\cup H_j$.

For Setup~\ref{setup2}, the proof proceeds analogously to that of Setup~\ref{setup1}. 
In the first step, we apply Lemma~\ref{lemma:3.2} with $m=t+1$ (in place of the supersaturation result) to find ordered $s(t+1)^2$-sets $Y$ such that $Y\cup S$ spans a copy of 
$H_{\text{det}}^2$ (see Definition~\ref{def:H_{det}^2}) as candidates for $S$-absorbers.
Subsequently, we apply Lemma~\ref{Phi:abs2} (in place of Lemma~\ref{Phi:abs1}) and Proposition~\ref{p1} to complete the embeddings of copies of $\mathcal{C}_W$ by adding the random edges.
\end{proof}

We are now prepared to establish property (ii) in Lemma~\ref{lem:absorbing}. 
For each family $\{S_{i,1},\dots,S_{i,t_i}\}$ with $1\leq i \leq q$ and $t_i \leq \gamma_1 n$, we find $\{T_{i,1},\dots,T_{i,t_i}\}$ step by step. 
Let $t_0=1$ and $T_{0,1} = \emptyset$. 
Suppose we have already found $\{T_{x,1}, \dots, T_{x,t_x}\}$ for all $0 \leq x \leq i-1$. 
We now find $\{T_{i,1},\dots,T_{i,t_i}\}$. 
Let $S^*=\bigcup_{i \in [q], j \in [t_i]} S_{i,j}\cup \bigcup_{0 \leq x \leq i-1}(T_{x,1} \cup \cdots \cup T_{x,t_x})$ and $X = V \setminus S^* $. 
By $q \leq \frac{120 \tau}{\gamma_1}$ and $\tau\ll \eta$, we have
\[
\left|S^*\right| \leq q \gamma_1  (t+1)n + q \gamma_1 s(t+1)^2 n \leq \eta^2 n.
\]
Thus, $|X| \geq n - \eta^2 n$.
By Claim~\ref{S}, a.a.s.~$(H\cup H_{i})[X]$ contains a family $\{T_{i,1},\dots,T_{i,t_i}\}$ such that, for every $j \in [t_i]$, $T_{i,j}$ is an $S_{i,j}$-absorber. 
Repeating this process for all $i \in [q]$ and by union bound, a.a.s.~we obtain pairwise  disjoint families $\{T_{1,1},\dots,T_{1,t_1}\}, \dots, \{T_{q,1},\dots,T_{q,t_q}\}$ such that, for every $i \in [q]$ and $j \in [t_i]$, $T_{i,j}$ is an $S_{i,j}$-absorber.

This concludes the proof. 
\end{proof}

\subsection{Verification of condition~$(i)$   of Lemma~\ref{lem:absorbing}}


\begin{lemma}\label{vertexcover}
Let $t\geq k\ge 3$ be integers, and suppose $1/n\ll \gamma_1\ll \varepsilon\ll \beta \ll \tau \ll \gamma \ll\alpha,1/k,1/t$.  
Let $H=(V,E)$ be an $n$-vertex $k$-graph with $\delta_{k-1}(H)\ge \alpha n$.
Then there exists a constant $C= C(\alpha,\varepsilon) > 0$ such that the following holds.
If $p\geq Cn^{-t/\binom{t+1}{k}}$, then $H^*:=H\cup H^{(k)}(n,p)$ contains a vertex set $V_1\subseteq V(H)$ of size $\tau n$ such that  for any subset $S\subseteq V(H^*)$ of size at most $\gamma_1 n$, a.a.s.~there exists a set $T\subseteq V_1\setminus S$ with $|T|=t|S|$ such that $H^*[S\cup T]$ contains a $K_{t+1}^{(k)}$-factor. 
In particular, Lemma \ref{lem:absorbing} $(i)$  holds for $H^*$ and $F=K_{t+1}^{(k)}$.
\end{lemma}

\begin{proof}
Note that $H$ has density at least $\alpha$.
By averaging, let $V_1\subseteq V(H)$ be a subset of size $\tau n$ with density at least $\alpha$.
Fix any subset $S\subseteq V(H)$ with $|S|=m\le \gamma_1 n$ and let $S:=\{s_1,\dots,s_{m}\}.$
For each $i\in[m]$, we aim to find a copy of $K_{t+1}^{(k)}$ in $V_1\cup\{s_i\}$ containing $s_i$, such that all these copies are vertex-disjoint and together form a $K_{t+1}^{(k)}$-tiling covering $S$.

Let $\mathcal F$ be a family of ordered $t$-sets in $V_1$ such that the first $k$ vertices span an edge of $H$.
By the density condition, we can take $|\mathcal{F}|\ge \alpha \tau^t n^t/2 \ge \beta n^t$.
%

Next, we complete the copies of $K_{t+1}^{(k)}$ by adding  edges from $H^{(k)}(n,p)$ to fill the missing edges by Proposition~\ref{p1}.
Let $K'$ be obtained by deleting any edge $e$ from $K_{t+1}^{(k)}$ and $W=\{w\}$ for $w\in V(K_{t+1}^{(k)})\setminus e$.
By Lemma~\ref{l} (ii), $\Phi_{K'\setminus W} \ge Cn$ and $\Phi_{K', W} \ge Cn^\varepsilon$.
We apply Proposition~\ref{p1} with $F=K', W=\{w\}, b=t, U_i=\{s_i\}$, and with the candidate families $\mathcal F_i=\mathcal F$, $i\in[m]$. 
It follows that  a.a.s.~there exist pairwise disjoint embeddings $\phi_1,\dots,\phi_{m}$ such that  each $\phi_{i}$ embeds  a copy of $K'$ into $H^{(k)}(n,p)$ on $V_1$ with $W$ mapped to $s_{i}$ and $V(K')\setminus W$ mapped to some $t$-set in $\mathcal{F}_{i}$.
Consequently, we obtain $m$ vertex-disjoint $t$-sets $T_1, \dots, T_{m}$ in
$V_1$ such that each $T_i\cup \{s_i\}$ spans a copy of $K_{t+1}^{(k)}$.
Let $T=\bigcup_{i\in[m]} T_i$.
Then $|T|=tm$, and a.a.s.~$H^*[S\cup T]$ contains a $K_{t+1}^{(k)}$-factor. 
\end{proof}

\section{Proof of the main theorem}\label{s5}
In this section, we complete the proof of Theorem~\ref{main}.
We first derive a covering lemma, which allows us to find a collection of vertex-disjoint copies of $K_{t+1}^{(k)}$ to cover all but a small fraction of vertices.  Then we combine this covering lemma with Lemma~\ref{lem:absorbing},   Lemma~\ref{lem:absa}  and  Lemma~\ref{vertexcover} to prove Theorem~\ref{main}.

\begin{lemma}[Covering Lemma]\label{lem:cov}
  Let $t\geq k\ge 3$  be integers. Let $\gamma>0$ and suppose $1/n \ll 1/C\ll \gamma\ll  1/k,1/t$.
  Let $H^{(k)}(n,p)$ be the binomial random $k$-graph on an $n$-vertex set $V$.
  If $p=p(n)\geq C n^{-t/\binom{t+1}{k}}$, then a.a.s.~$H^{(k)}(n,p)$ contains a collection of vertex-disjoint copies of $K_{t+1}^{(k)}$ covering all but at most $\gamma n$ vertices of $V$.
\end{lemma}

Our proof of Lemma~\ref{lem:cov} uses Lemma~\ref{lem:PropOfPhi}~\ref{fact:phiK_{t+1}} and the following result from~\cite{BHKM2019}, which ensures that $H^{(k)}(n,p)$ can be almost
covered by a collection of vertex-disjoint copies of $K_{t+1}^{(k)}$.

\begin{lemma}\cite[Lemma 2.2 (i)]{BHKM2019}
\label{l5}
Let $F$ be a labelled $k$-graph with $b$ vertices.
Let $\gamma>0$ and suppose $1/n\ll 1/C\ll \gamma,1/b$.
Let $H^{(k)}(n,p)$ be the binomial random $k$-graph on an $n$-vertex set $V$.
If $p=p(n)$ is such that $\Phi_{F}\geq Cn$, then a.a.s.~every induced subgraph of $H^{(k)}(n,p)$ of order $\gamma n$ contains a copy of $F$.
\end{lemma}

Now we are ready to prove Lemma~\ref{lem:cov}.

\begin{proof}[Proof of Lemma~\ref{lem:cov}]
Fix $t\geq k\ge 3$. 
Applying Lemma~\ref{lem:PropOfPhi}~\ref{fact:phiK_{t+1}} and Lemma~\ref{l5} to $H^{(k)}(n,p)$ with $F=K_{t+1}^{(k)}$ implies that a.a.s.~every induced subgraph of $H^{(k)}(n,p)$ with at least $\gamma n$ vertices contains a copy of $K_{t+1}^{(k)}$.
Consequently, we can greedily find pairwise vertex-disjoint copies of $K_{t+1}^{(k)}$ from $V$ until fewer than $\gamma n$ vertices remain.
\end{proof}

We now combine Lemma~\ref{lem:cov} with the absorbing results established in Section~4 to prove the main theorem.

\begin{proof}[Proof of Theorem~\ref{main}]
Let $t\geq k\ge 3$ be fixed integers, and let $n\in (t+1)\mathbb{N}$ be sufficiently large. 
Assume Setup~\ref{setup1} or Setup~\ref{setup2} holds.
Suppose $\alpha> 0$ and  $1/n\ll 1/C\ll\xi\ll \gamma\ll \alpha,1/k,1/t$.
Let $H=(V,E)$ be an $n$-vertex $k$-graph with $\delta_{k-1}(H)\geq \alpha n$, and
suppose $p=p(n)\geq 2C n^{-t/\binom{t+1}{k}}$.
Expose $H'=H^{(k)}(n,p)$ in two rounds: $H' = H_1 \cup H_2 $, where $H_1$ and $H_2$ are independent copies of $H^{(k)}(n, p')$, with $(1-p')^2=1-p$.
Note that $(1-p')^2>1-2p'$, which implies $p'> p/2\geq Cn^{-t/\binom{t+1}{k}}$.

By Lemma~\ref{lem:absa} and Lemma~\ref{vertexcover}, $H^{*}:=H\cup H_1$ satisfies the conditions $(i)$ and $(ii)$ of Lemma~\ref{lem:absorbing}.
Consequently, $H^{*}$ contains an absorbing set $V_A\subseteq V(H)$ with $|V_A|\le\gamma n$ such that, for every  $U\subseteq V\setminus V_A$ with $|U|\le \xi n$ and $(|V_A|+|U|)\in(t+1)\mathbb{N}$, a.a.s.~$H^{*}[V_A\cup U]$ contains a $K_{t+1}^{(k)}$-factor.
Next, we apply Lemma~\ref{lem:cov} to $H_2$ 
to obtain a collection of vertex-disjoint copies of $K_{t+1}^{(k)}$ that together cover all but at most $\xi n$ vertices of $V\setminus V_A$.
Let $R$ denote the set of uncovered vertices.
Then $|R|\le \xi n$.
Since $n$ is divisible by $t+1$, and the set $V \setminus (V_A\cup R)$ is covered by vertex-disjoint copies of $K_{t+1}^{(k)}$, it follows that $|V_A| + |R|$ is divisible by $t+1$.
Since $|R| \le \xi n$, a.a.s.~$H^{*}[V_A\cup R]$ contains a $K_{t+1}^{(k)}$-factor.
Combining this with the copies covering $V \setminus (V_A \cup R)$, we a.a.s.\ obtain the desired $K_{t+1}^{(k)}$-factor in $H\cup H^{(k)}(n,p)$.
\end{proof}

\section{Concluding Remarks}

In this paper we studied the clique-factor problem in randomly perturbed hypergraphs under (vanishing) minimum codegree conditions.
Our main result, Theorem~\ref{main}, resolves the problem for all sufficiently large cliques, while we believe that the result should hold for all clique-factors, that is, for all $t\ge k-1$.

\begin{conjecture}
Let $k \ge 3$ and let $t\ge k-1$.
For every $\alpha>0$, there exists $C>0$ such that the following holds for sufficiently large $n\in(t+1)\mathbb{N}$.
Let $H$ be a $k$-graph on $n$ vertices with $\delta_{k-1}(H)\geq \alpha n$, and let $p= p(n)\geq Cn^{-{t/\binom{t+1}{k}}}$.
Then a.a.s.~$H\cup H^{(k)}(n,p)$ contains a $K_{t+1}^{(k)}$-factor.
\end{conjecture}

It is also possible that one can replace the minimum codegree to the much weaker minimum vertex degree, and an even broader conjecture is to extend the results to $F$-factors for arbitrary $k$-graph $F$.
Our current approach fails for small cliques and
the limitation is the construction of the absorbers, based on the idea of \textit{reachable sets}, where indeed the 1-reachable sets --  $K_{t+1}^{(k)}$ with one vertex duplicated are expensive for (sparse) random $k$-graphs.
In our proof, we compensate this loss by forcing complete $k$-partite $k$-graphs by deterministic edges using the minimum codegree conditions, which are efficient for sufficiently large cliques.

\vspace{2mm}

\noindent
\textbf{AI Disclosure.} All ideas  are our own and were  developed long before LLMs were available. The writing is also completely ours and we used ChatGPT 6 only to do a final sweep for small typos and mathematical errors.

\bibliographystyle{abbrv}
\bibliography{clique_factors__perturbed}

\end{document}